\documentclass[12pt]{article}

\usepackage{amssymb}
\usepackage{amsmath}
\usepackage{amsthm}
\usepackage{amsfonts}
\usepackage{color}

\newcommand{\mmp}{\mathbb{P}}
\newcommand{\od}{\overset{d}{=}}
\newcommand{\dod}{\overset{d}{\to}}
\newcommand{\tp}{\overset{P}{\to}}

\newcommand{\me}{\mathbb{E}}
\newcommand{\mr}{\mathbb{R}}
\newcommand{\mn}{\mathbb{N}}
\newcommand{\mno}{\mathbb{N}_0}

\newtheorem{thm}{Theorem}[section]
\newtheorem{lemma}[thm]{Lemma}

\newtheorem{assertion}[thm]{Proposition}
\theoremstyle{definition}

\theoremstyle{remark}

\begin{document}
\title{Weak convergence of the number of zero increments in the random walk with barrier}
\date{}
\author{Alexander Marynych\footnote{Faculty of Cybernetics, Taras Shevchenko National University of Kyiv, Ukraine. E-mail:
marynych@unicyb.kiev.ua} \and Glib Verovkin\footnote{Faculty of Mechanics and Mathematics, Taras Shevchenko National University of Kyiv, Ukraine. E-mail: glebverov@gmail.com}} %

\maketitle
\begin{abstract}
\noindent
We continue the line of research of random walks with barrier initiated by Iksanov and M{\"o}hle (2008). Assuming that the tail of the step of the underlying random walk has a power-like behavior at infinity with exponent $-\alpha$, $\alpha\in(0,1)$, we prove that the number $V_n$ of zero increments in the random walk with barrier, properly centered and normalized, converges weakly to the standard normal law. This refines previously known weak law of large numbers for $V_n$ proved in Iksanov and Negadailov (2008).
\end{abstract}
\noindent {\em Keywords}:  random walk with barrier, recursion with random indicies, renewal process, undershot


\section{Introduction}

Let $(\xi_k)_{k\in\mn}$ be independent copies of a random variable
$\xi$ with distribution $p_k=\mmp\{\xi=k\}$, $k\in\mn$. The random walk with barrier
$n\in\mn$ is a sequence $(R_k^{(n)})_{k\in\mn_0}$ (where $\mn_0:=\mn\cup \{0\}$) defined as
follows:
$$R_0^{(n)}:=0 \ \ \text{and} \ \
R_k^{(n)}:=R_{k-1}^{(n)}+\xi_k 1_{\{R_{k-1}^{(n)}+\xi_k<n\}}, \ \
k\in\mn.$$ Plainly, $(R_k^{(n)})_{k \in \mn_0}$ is a
non-decreasing Markov chain which cannot reach the state $n$. In
what follows we always assume that $p_1>0$ which implies that the random walk with barrier $n$ will
eventually get absorbed in the state $n-1$.

The equalities
$$M_n:= \# \{k \in \mn : R_{k-1}^{(n)} \neq R_{k}^{(n)} \} = \sum_{l=0}^{\infty}1_{\{R_{l}^{(n)} + \xi_{l+1} < n\}};$$
$$T_n:=\inf\{k\in\mn_0: R_k^{(n)}=n-1\}=\sum_{l\geq 0} 1_{\{R_l^{(n)}<n-1\}};$$
$$V_n: = T_n - M_n = \#\{i\leq T_n: R_{i-1}^{(n)}=
R_i^{(n)}\}=\sum_{l=0}^{T_n-1}1_{\{R_l^{(n)}+\xi_{l+1}\geq n\}}$$
define, respectively, the number of jumps, the absorption time and
the number of zero increments before the absorption in the random
walk with barrier $n$.

There is a large number of real life situations where the random walk
with barrier appears naturally. Let PTC be a transport company, offering a tour to the 
national park. The PTC uses buses with total amount of seats $n$. Various groups of 
people book seats in order to visit the park. If the size of the group is less than remaining 
number of vacant seats, the request satisfied, otherwise it is turned down. The quantities 
of interest are the total number of groups applied $T_{n+1}$,
the number of accepted groups $M_{n+1}$ and the number of rejections $V_{n+1}$.

Another example is the work of a server. Imagine that a client has bought an internet-package $n$ Mb in size.
Consider the downloading of files with the size being a multiple of 1 Mb: the server receives requests on 
download, if the size of file is lower than remaining
size, then it starts downloading it, else blocks the request. Similarly to the example above, the
quantities of interest in this case are the the total number of requests $T_{n+1}$,
the number of downloaded files $M_{n+1}$ and the number of blocked requests $V_{n+1}$.

In \cite{IksMoe2} (see also \cite{Haas} for a particular case) it
was shown that, if the law of $\xi$ belongs to the
domain of attraction of a stable law, $M_n$, properly normalized
and centered, weakly converges. Furthermore, the set of limiting
laws is comprised of stable laws and the law of exponential
subordinator. In \cite{Negad} it was checked that the same group
of results hold on replacing $M_n$ by $T_n$. Finally, in
\cite{IksNeg2} it was proved that: (a) if $\me \xi<\infty$ then
$V_n$ weakly converges (without normalization); (b) if the law of
$\xi$ belongs to the domain of attraction of an $\alpha$-stable
law with $\alpha\in (0,1]$, equivalently  if
\begin{equation}\label{reg}
\mmp\{\xi\geq n\}\sim n^{-\alpha}\ell(n),\;\;n\to\infty,
\end{equation}
for some $\ell$ slowly varying at infinity, then $V_n/\me V_n \tp 1$ as $n\to\infty$.

To complete the picture, in this paper we give results about the
weak convergence of $V_n$. The treatment of $V_n$ calls for more delicate argument than that for
$M_n$ and/or $T_n$. Crudely speaking, while the asymptotics of
$M_n$ and $T_n$ is based on the "first order"\, arguments, the
asymptotics of $V_n$ needs the "second order"\, reasoning. As a
result, the approach exploited in \cite{IksMoe2, IksNeg2} does not
help in the present situation. Moreover, regular variation \eqref{reg} alone
seems not to be enough to ensure the weak convergence of properly scaled
and normalized $V_n$ and one has to impose more restrictive "second-order" condition
on the tail $\mmp\{\xi\geq n\}$. In this work we prove a central limit theorem-type result for $V_n$ assuming

\begin{equation}\label{main_assumption}
\mmp\{\xi\geq n\}=cn^{-\alpha} +O(n^{-(\alpha+\varepsilon)}),\;\;n\to\infty,
\end{equation}
for some $c>0$, $\alpha\in(0,1)$ and $\varepsilon>0$.

In what follows we reserve notation $\eta$ for a random
variable with the beta $(1-\alpha, \alpha)$ law, $\alpha\in
(0,1)$, i.e.,
\begin{equation}\label{beta_density}
\mmp\{\eta\in {\rm d}x\}= \frac{\sin\pi\alpha}{\pi}x^{-\alpha}(1-x)^{\alpha-1}1_{(0,1)}(x){\rm d}x;
\end{equation}
$$\mu_\alpha:=\me |\log \eta|=\psi(1)-\psi(1-\alpha)$$ and
$$\sigma^2_\alpha:={\rm Var}\,(\log
\eta)=\psi^\prime(1-\alpha)-\psi^\prime(1),$$ where
$\psi(x)=\Gamma^\prime(x)/\Gamma(x)$ is the logarithmic derivative
of the gamma function.

The main result of this paper is given by the next theorem
\begin{thm}\label{main_thm}
Assume that \eqref{main_assumption} holds with $\alpha\in(0,1)$, $\varepsilon>0$ and $c>0$. If $\alpha\in(0,1/2]$ assume additionaly
\begin{equation}\label{doney_assumption}
\sup_{n\geq 1}\frac{np_n}{\mmp\{\xi>n\}}<\infty.
\end{equation}
 Then
\begin{equation*}
\frac{V_n-\mu_{\alpha}^{-1}\log n}{\sqrt{\sigma_{\alpha}^2\mu_{\alpha}^{-3}\log n}}\dod \mathcal{N}(0,1),\;\;n\to\infty,
\end{equation*}
where $\mathcal{N}(0,1)$ is a random variable with the standard normal law. Moreover, there is a convergence of the first absolute moments.
\end{thm}

Our approach is based on the analysis of random recursive equation for $(V_n)$. It is shown that the sequence $(V_n)$ can be approximated by
a suitable renewal counting process and the error of such an approximation is estimated in terms of an appropriate probability distance. A similar method has already been used in \cite{GIMM} to derive the weak convergence result for the number of collisions in beta coalescents.

The rest of the paper is organized as follows. In Section \ref{recursions} we define the approximating renewal process and give random recursive equations for related quantities. The proofs are presented in Section \ref{proofs_sec}. An auxiliary lemma is formulated and proved in Appendix.

\section{Renewal process and recursion with random indicies}\label{recursions}
Given the sequence $(\xi_n)_{n\in\mn}$, define a zero-delayed random walk
$$
S_0=0,\;\;S_n=\xi_1+\ldots+\xi_n,\;\;n\in\mn,
$$
and the first passage process
$$
N_n:=\inf\{k\in\mno : S_k\geq n\},\;\;n\in\mn.
$$
The random variable $Y_n:=n-S_{N_n-1}$ is called {\it undershot}. It was shown\footnote{Note that in \cite{IksNeg2} the definition of $T_n$ is slightly different from our which results in different recursion for $(V_n)$.} in \cite{IksNeg2} that the sequence $(V_n)_{n\in\mn}$ satisfies
the following recursion with random index
\begin{equation}\label{rec_zi}
V_1=0,\;\;V_n\od 1_{\{{Y_n > 1}\}} + V'_{Y_n},\;\;n\geq 2,
\end{equation}
where $V'_k\od V_k$ for all $k\in\mn$ and $(V'_k)_{k\in\mn}$ and $Y_n$ are independent.

The recursion \eqref{rec_zi} can be slightly simplified by setting $X_n:=V_n+1_{\{n>1\}}$, then
\begin{equation}\label{rec_zim}
X_1=0,\;\;X_n\od 1+X'_{Y_n},\;\;n\geq 2,
\end{equation}
where likewise $X'_k\od X_k$ for all $k\in\mn$ and $(X'_k)_{k\in\mn}$ and $Y_n$ are independent.
Clearly, the asymptotic behavior of $X_n$ is the same as of $V_n$.

It is a classical observation due to Dynkin \cite{Dynkin} that under the assumption \eqref{reg} with $\alpha\in(0,1)$ we have
\begin{equation}\label{u_conv}
Y_n/n\dod \eta, \;\;n\to\infty,
\end{equation}
where $\eta$ has density \eqref{beta_density}.

Let $(\eta_k)_{k\in\mn}$ be iid copies of $\eta$. Define a zero-delayed random walk
$$
T_0=0;\;\;T_k=|\log \eta_1|+\ldots+|\log \eta_k|,\;\;k\in\mn;
$$
the corresponding renewal counting process
$$
\nu_t:=\#\{k\in\mn:T_k \leq t\}=\sum_{k=1}^{\infty}1_{\{T_k\leq t\}},\;\;t\in\mr,
$$
and set $W_t:=\nu_{\log t}+1_{\{t>1\}}$ for $t>0$. Since $\nu_t=0$ a.s. for $t\leq 0$ we have $W_t=0$ for $t\in(0,1]$, while for $t>1$ the strong Markov property implies
\begin{equation}\label{renewal_rec}
W_t\od 1+ W'_{t\eta},
\end{equation}
where $W_t\od W'_t$ for every $t>0$ and $(W'_t)_{t\geq 0}$ and $\eta$ are independent.

Comparing recursions \eqref{rec_zim} and \eqref{renewal_rec} and in view of \eqref{u_conv} we may expect that the weak asymptotic behavior of $X_n$ is the same as of $W_n$. We will show, assuming \eqref{main_assumption}, that this heuristic can be made rigorous and leads to the desired result on the asymptotic of $V_n$.

\section{Proofs}\label{proofs_sec}
We start with a refinement of \eqref{u_conv} by estimating the speed of convergence of $Y_n/n$ to $\eta$ in terms of
so-called minimal $L_1$-distance. Let us recall its definition. Let $\mathcal{D}_{1}$ be the set of probability laws on ${\mathbb R}$ with finite
first absolute moment.
The $L_1$-minimal (or Wasserstein) distance on  $\mathcal{D}_{1}$ is defined by
\begin{equation}\label{was_dis_rv}
d_{1}(X,Y)=\inf\me|\widehat{X}-\widehat{Y}|,
\end{equation}
where the infimum is taken over all couplings $(\widehat{X},\widehat{Y})$ such that
$X\od \widehat{X}$ and $Y\od\widehat{Y}$.

For ease of reference we summarize the properties of $d_{1}$ to be used in this work in the following proposition.

\begin{assertion}\label{was_dis_prop}
Let $X,Y$ be  random variables with finite first absolute moments.
The distance $d_1$ has the following properties:
 \begin{itemize}
\item[\rm(Int)] $d_1(X,Y)$ has an integral representation:
$$
d_1(X,Y):=\int_{\mr}|\mmp\{X\leq x\}-\mmp\{Y\leq x\}|{\rm d}x.
$$
\item[\rm(Rep)] $d_1(X,Y)$ has a dual representation:
$$
d_{1}(X,Y)=\sup_{f\in\mathcal{F}}|\me f(X)-\me f(Y)|.
$$
where $\mathcal{F}:=\{f\;:\; |f(x)-f(y)|\leq |x-y|\}$,
\item[\rm(Lin)] $d_{1}(cX+a,cY+a)=|c|d(X,Y)$ for  $a,c\in\mr$.
\item[\rm(Conv)] For $X, X_n\in\mathcal{D}_{1}$ convergence $d_{1}(X_n,X)\to 0$, $n\to\infty$, is equivalent to
$X_n\dod X$ and $\me|X_n|\to\me|X|$, $n\to\infty$.
\end{itemize}
\end{assertion}

We refer the reader to Chapter 1 in \cite{Zolotarev} for an introduction to the theory of probability metrics,
in particular for the proofs of the aforementioned properties of $d_1$.

In view of (Conv) characterization  of $d_1$ the next lemma is indeed a refinement of \eqref{u_conv}.
\begin{assertion}\label{speed_in_u_conv}
Under the assumptions of Theorem \ref{main_thm} there exists $\delta>0$ such that
\begin{equation*}
d_1\Big(\log \frac{Y_n}{n},\log \eta\Big)=d_1\Big(\log Y_n,\log (n\eta)\Big)=O(n^{-\delta}),\;\;n\to\infty.
\end{equation*}
\end{assertion}
\begin{proof}
The first equality follows from (Lin) property of $d_1$. Using (Rep) we have
\begin{equation}\label{kr_rep}
d_1\Big(\log Y_n,\log (n\eta)\Big)=\sup_{f\in\mathcal{F}_1}\Big|\me f(\log Y_n)-\me f(\log (n\eta))\Big|.
\end{equation}
From the distributional identity
$$
Y_1=1,\;\;Y_n\od n1_{\{\xi\geq n\}}+Y'_{n-\xi}1_{\{\xi<n\}},\;\;n\geq 2,
$$
where $Y'_k\od Y_k$ for all $k\in\mn$ and $(Y'_k)_{k\in\mn}$ is independent from $\xi$, we infer
$$
\me f(\log Y_n)=\mmp\{\xi\geq n\}f(\log n)+\sum_{j=1}^{n-1}p_j\me f(\log Y_{n-j}),\;\;n\geq 2.
$$
Substituting this into \eqref{kr_rep} and using the triangle inequality gives
\begin{eqnarray*}
&&\hspace{-1.5cm}d_1\Big(\log Y_n,\log (n\eta)\Big)\\
&&\hspace{-1cm}\leq \sup_{f\in\mathcal{F}_1}\Big|\mmp\{\xi\geq n\}f(\log n)+\sum_{j=1}^{n-1}p_j\me f(\log (n-j)\eta)-\me f(\log (n\eta))\Big|\\
&&\hspace{5cm}+\sum_{j=1}^{n-1}p_j\sup_{f\in\mathcal{F}_1}\Big|\me f(\log Y_{n-j})-\me f(\log (n-j)\eta)\Big|\\
&&\hspace{-1cm}= \sup_{f\in\mathcal{F}_1}\Big|\mmp\{\xi\geq n\}f(\log n)+\sum_{j=1}^{n-1}p_j\me f(\log (n-j)\eta)-\me f(\log (n\eta))\Big|\\
&&\hspace{5cm}+\sum_{j=1}^{n-1}p_jd_1\Big(\log Y_{n-j},\log (n-j)\eta\Big).
\end{eqnarray*}
Let $\tilde{\xi}$ be independent of $\tilde{\eta}$ and $\tilde{\xi}\od\xi$, $\tilde{\eta}\od \eta$. The first term can be written as
\begin{eqnarray*}
&&\hspace{-0.5cm}\sup_{f\in\mathcal{F}_1}\Big|\mmp\{\xi\geq n\}f(\log n)+\sum_{j=1}^{n-1}p_j\me f(\log (n-j)\eta)-\me f(\log (n\eta))\Big|\\
&&=d_1\Big(\log (n1_{\{\tilde{\xi}\geq n\}}+(n-\tilde{\xi})\tilde{\eta} 1_{\{\tilde{\xi}<n\}}),\log(n\tilde{\eta})\Big)=d_1\Big((\log (1-\tilde{\xi} n^{-1})\tilde{\eta}) 1_{\{\tilde{\xi}<n\}},\log\tilde{\eta}\Big),
\end{eqnarray*}
where we have utilized (Lin) property of $d_1$ in the second equality.

For every $x\geq 1$,
$$
\mmp\{\xi\geq x\}=\mmp\{\xi\geq \lceil x\rceil\}=c(\lceil x\rceil)^{-\alpha}+O((\lceil x\rceil)^{-(\alpha+\varepsilon)})=cx^{-\alpha}+O(x^{-((\alpha+\varepsilon)\wedge 1)}),
$$
hence, by Lemma \ref{1111} with $\beta=1$ and $x=n$, there exist $K>0$ and $\delta\in (0,1-\alpha)$ such that
$$
d_1\Big(\log Y_n,\log (n\eta)\Big)\leq Kn^{-(\alpha+\delta)}+\sum_{j=1}^{n-1}p_jd_1\Big(\log Y_{n-j},\log (n-j)\eta\Big).
$$
Using 1-arithmetic variant of Theorem 1 in \cite{AndersonAthreya} and also Theorem B in \cite{Doney} if $\alpha\in(0,1/2]$ (see also Theorem 1 in \cite{GarsiaLamperti}), we obtain
$$
d_1\Big(\log Y_n,\log (n\eta)\Big)=O(n^{-\delta}),\;\;n\to\infty.
$$
The proof is complete.
\end{proof}

\subsection{Proof of Theorem \ref{main_thm}}\label{proof_2}
It is enough to prove Theorem \ref{main_thm} for $V_n$ replaced by $X_n$. In view of (Conv) property of $d_1$, in order to prove Theorem \ref{main_thm} we need to check
$$
d_1\Big(\frac{X_n-\mu_a^{-1}\log n}{\sqrt{\sigma_a^2\mu_{a}^{-3}\log n}}, \mathcal{N}(0,1)\Big)\to 0,\;\;n\to\infty.
$$
Using the triangle inequality yields for $n\geq 2$,
\begin{eqnarray*}
d_1\Big(\frac{X_n-\mu_a^{-1}\log n}{\sqrt{\sigma_a^2\mu_{a}^{-3}\log n}}, \mathcal{N}(0,1)\Big)&\leq& d_1\Big(\frac{X_n-\mu_a^{-1}\log n}{\sqrt{\sigma_a^2\mu_{a}^{-3}\log n}}, \frac{W_n-\mu_a^{-1}\log n}{\sqrt{\sigma_a^2\mu_{a}^{-3}\log n}}\Big)\\
&&\hspace{3cm}+d_1\Big(\frac{W_n-\mu_a^{-1}\log n}{\sqrt{\sigma_a^2\mu_{a}^{-3}\log n}}, \mathcal{N}(0,1)\Big)\\
&=&d_1\Big(\frac{X_n-\mu_a^{-1}\log n}{\sqrt{\sigma_a^2\mu_{a}^{-3}\log n}}, \frac{W_n-\mu_a^{-1}\log n}{\sqrt{\sigma_a^2\mu_{a}^{-3}\log n}}\Big)\\
&&\hspace{3cm}+d_1\Big(\frac{\nu_{\log n}+1-\mu_a^{-1}\log n}{\sqrt{\sigma_a^2\mu_{a}^{-3}\log n}}, \mathcal{N}(0,1)\Big)
\end{eqnarray*}

The second term converges to zero in view of the CLT for the renewal process with finite variance (see Chapter XI.5 in \cite{Feller}) as well as the convergence of first absolute moments (see Proposition A.1 in \cite{IksMarMei1}).
From (Lin) property of $d_{1}$ we see that it is enough to prove
\begin{equation}\label{o_est_rv1}
d_{1}(X_n,W_n)=O(1),\;\;n\to\infty.
\end{equation}

Using the recursions for $X_n$ and $W_n$ we have, in view of (Lin) property of $d_{1}$,
\begin{eqnarray*}
t_n:=d_{1}(X_n,W_n) &=& d_{1}(X'_{Y_n},W'_{n\eta})\leq d_{1}(W'_{n\eta},W'_{Y_n})\\
&+&d_{1}(W'_{Y_n},X'_{Y_n})\leq d_{1}(W'_{n\eta},W'_{Y_n}) + \me|\widehat{W}_{Y_n}-\widehat{X}_{Y_n}|\\
&=:&c_n+\sum_{k=2}^{n}\mmp\{Y_n=k\} \me|\widehat{X}_{k}-\widehat{W}_{k}|,
\end{eqnarray*}
for arbitrary pairs $\{(\widehat{X}_k,\widehat{W}_k):2\leq k\leq n\}$ independent of $Y_n$ such that $\widehat{X}_k\od X_k$, $\widehat{W}_k\od W_k$.
Passing to infimum over all such pairs in both sides of inequality leads to
\begin{equation}\label{recursion_rv1}
  t_n\leq c_n+\sum_{k=2}^{n}\mmp\{Y_n=k\}t_k.
\end{equation}
In order to estimate $c_n$ we proceed as follows. Let
$(\hat{Y}_n,\hat{\eta})$ be a coupling of $Y_n$ and $\eta$ such that
$d_{1}(\log Y_n,\log (n\eta))=\me|\log\hat{Y}_n-\log (n\hat{\eta})|$.
Let $(\hat{\nu}_t)_{t\in\mr}$ be a copy of $(\nu_t)_{t\in\mr}$ independent of $(\hat{Y}_n,\hat{\eta})$. We have
\begin{eqnarray*}
c_n=d_{1}(W'_{Y_n},W'_{n\eta(a)})&=&d_1(\hat{\nu}_{\log \hat{Y}_n}+1_{\{\hat{Y}_n>1\}},\hat{\nu}_{\log (n\hat{\eta})}+1_{\{n\hat{\eta}>1\}})\\
&\leq&\me |\hat{\nu}_{\log \hat{Y}_n}+1_{\{\hat{Y}_n>1\}}-\hat{\nu}_{\log (n\hat{\eta})}-1_{\{n\hat{\eta}>1\}}|\\
&\leq&\me |\hat{\nu}_{\log \hat{Y}_n}-\hat{\nu}_{\log (n\hat{\eta})}|+\mmp\{Y_n=1\}+\mmp\{n\eta\leq 1\}
\end{eqnarray*}
where the penultimate inequality follows from the definition of $d_1$, since $(\hat{Y}_n,\hat{\eta},(\hat{\nu}(t)))$ is a particular coupling. There exists $\rho>0$ such that the last two summands are $O(n^{-\rho})$. To bound the first term, we apply the distributional subadditivity of $(\nu_t)$:
$$
\nu_{x+y}-\nu_{x}\stackrel{d}{\leq}\nu_y,\;\;x,y\in\mr,
$$
which yields
\begin{equation}\label{c_n_est_2}
c_n\leq \me \hat{\nu}_{|\log \hat{Y}_n-\log (n\hat{\eta})|}+O(n^{-\rho}).
\end{equation}
Note that for every $x\geq 0$,
$$
\mmp\{T_1\leq x\}\leq \me\nu_x=\sum_{k=1}^{\infty}\mmp\{T_k\leq x\}\leq \sum_{k=1}^{\infty}(\mmp\{T_1\leq x\})^k=\frac{\mmp\{T_1\leq x\}}{\mmp\{T_1>x\}},
$$
hence, by the standard sandwich argument,
$$
\lim_{x\downarrow 0}\frac{\me\nu_x}{x^{\alpha}}=\frac{\sin\pi\alpha}{\pi\alpha}.
$$
On the other hand, from the elementary renewal theorem we have
$$
\lim_{x\to\infty}\frac{\me\nu_x}{x}=\frac{1}{\me T_1},
$$
therefore there exist constants $c_1,c_2>0$ such that for all $x\geq 0$,
\begin{equation}\label{ren_func_est}
\me\nu_x\leq c_1x^{\alpha}+c_2x.
\end{equation}
Using \eqref{ren_func_est} and \eqref{c_n_est_2} we obtain
\begin{eqnarray*}
c_n&\leq& c_1\me |\log \hat{Y}_n-\log (n\hat{\eta})|^{\alpha}+c_2 \me|\log \hat{Y}_n-\log (n\hat{\eta})|+O(n^{-\rho})\\
&\leq&c_1 d_1^{\alpha}(\log Y_n,\log (n\eta))+c_2 d_1(\log Y_n,\log (n\eta))+O(n^{-\rho}).
\end{eqnarray*}
By Lemma \ref{speed_in_u_conv} we conclude $c_n=O(n^{-\rho'})$ for some $\rho'>0$ as $n\to\infty$.

It remains to apply Lemma A.1 from \cite{GneIksMar_coal} with $\phi_n\equiv 1$ to \eqref{recursion_rv1} to conclude that
$$
t_n=O\Big(\sum_{k=1}^{n}\frac{k^{-\rho'}}{k}\Big)=O(1),\;\;n\to\infty.
$$
The proof of Theorem \ref{main_thm} is complete.

\section{Appendix}
The next lemma is the main ingredient in the proof of Proposition \ref{speed_in_u_conv}.
\begin{lemma}\label{1111}
Assume that $\theta$ is a random variable on $[1,+\infty)$ such that for some $c>0$, $\alpha\in(0,1)$ and $\varepsilon>0$
\begin{equation}\label{111}
1-F_{\theta}(x):=\mmp\{\theta\geq x\}=cx^{-\alpha}+O(x^{-(\alpha+\varepsilon)}),\;\;x\to\infty.
\end{equation}
Let $\eta$ be a random variable with density \eqref{beta_density} independent of $\theta$. Then for every $\beta>0$ there exists $\delta>0$ such that
\begin{equation}\label{11}
d_1\Big(\log ((1-\theta x^{-1})\eta) 1_{\{\theta<x-\beta\}},\log\eta\Big)=O(x^{-(\alpha+\delta)}),\;\;x\to\infty.
\end{equation}
\end{lemma}

\begin{proof}
Denote the left-hand side of \eqref{11} by $s_{\theta}(x,\beta)$. In view of relations
$$
s_{\theta}(x,\beta)=s_{c^{-1/\alpha}\theta}(c^{-1/\alpha}x,c^{-1/\alpha}\beta),\;\;x\geq 1,
$$
and
$$
\mmp\{c^{-1/\alpha}\theta\geq x\}=x^{-\alpha}+O(x^{-(\alpha+\varepsilon)}),\;\;x\to\infty,
$$
it is enough to prove the result for $c=1$. Fix $\beta$ for the rest of the proof. Using representation (Int) from Proposition \ref{was_dis_prop} we have
\begin{eqnarray*}
s_{\theta}(x,\beta)&=&\int_{-\infty}^{0}|\mmp\{\log (1_{\{\theta\geq x-\beta\}}+(1-\theta x^{-1})\eta 1_{\{\theta<x-\beta\}})\leq z\}-\mmp\{\log \eta\leq z\}|{\rm d}z\\
&=&\int_0^1 |{\mmp\{1_{\{\theta\geq x-\beta\}}+(1-\theta x^{-1})\eta 1_{\{\theta<x-\beta\}}\leq z\}}-\mmp\{\eta\leq z\}|z^{-1}{\rm d}z.
\end{eqnarray*}
Integrating by parts the first probability in the integrand, we obtain for $z\in [0,1)$ and $x>1+\beta$,
\begin{eqnarray*}
&&\hspace{-2.5cm}\mmp\{1_{\{\theta\geq x-\beta\}}+(1-\theta x^{-1})\eta 1_{\{\theta<x-\beta\}}\leq z\}\\
&&=-\int_{[1,x-\beta)}\mmp\{(1-y x^{-1})\eta\leq z\}{\rm d}(1-F_{\theta}(y))\\
&&=-\mmp\{\eta\leq \beta^{-1}xz\}\Big(1-F_{\theta}((x-\beta)-)\Big)+\mmp\{\eta\leq zx(x-1)^{-1}\}\\
&&+\int_{[1,x-\beta)}(1-F_{\theta}(y)){\rm d}_y\mmp\{(1-y x^{-1})\eta\leq z\}.
\end{eqnarray*}
Let $\theta_{\alpha}$ be a random variable independent of $\eta$ and with distribution
$$
1-F_{\theta_{\alpha}}(x):=\mmp\{\theta_{\alpha}\geq x\}=x^{-\alpha},\;\;x\geq 1.
$$
By the same reasoning as above,
\begin{eqnarray*}
&&\hspace{-2.6cm}\mmp\{1_{\{\theta_{\alpha}\geq x-\beta\}}+(1-\theta_{\alpha} x^{-1})\eta 1_{\{\theta_{\alpha}<x-\beta\}}\leq z\}\\
&&=-\int_{[1,x-\beta)}\mmp\{(1-y x^{-1})\eta\leq z\}{\rm d}(1-F_{\theta_{\alpha}}(y))\\
&&=-\mmp\{\eta\leq \beta^{-1}xz\}\Big(1-F_{\theta_{\alpha}}((x-\beta))\Big)+\mmp\{\eta\leq zx(x-1)^{-1}\}\\
&&+\int_{[1,x-\beta)}(1-F_{\theta_{\alpha}}(y)){\rm d}_y\mmp\{(1-y x^{-1})\eta\leq z\}.
\end{eqnarray*}
Subtracting the corresponding equations and using \eqref{111} we have for $z\in[0,1)$ and $x>1+\beta$,
\begin{eqnarray*}
\Big|\mmp\{1_{\{\theta\geq x-\beta\}}+(1-\theta x^{-1})\eta 1_{\{\theta<x-\beta\}}\leq z\}-\mmp\{1_{\{\theta_{\alpha}\geq x-\beta\}}+(1-\theta_{\alpha} x^{-1})\eta 1_{\{\theta_{\alpha}<x-\beta\}}\leq z\}\Big|\\
\leq K\Big(\mmp\{\eta\leq \beta^{-1}xz\}(x-\beta)^{-(\alpha+\varepsilon)}+\int_{[1,x-\beta)}y^{-(\alpha+\varepsilon)}{\rm d}_y\mmp\{(1-y x^{-1})\eta\leq z\}\Big),
\end{eqnarray*}
for some $K>0$ which does not depend on $x$ and $z$. Therefore,
\begin{eqnarray*}
&&\hspace{-1cm}s_{\theta}(x,\beta)\\
&&\leq\int_0^1 |{\mmp\{1_{\{\theta_{\alpha}\geq x-\beta\}}+(1-\theta_{\alpha} x^{-1})\eta 1_{\{\theta_{\alpha}<x-\beta\}}\leq z\}}-\mmp\{\eta\leq z\}|z^{-1}{\rm d}z\\
&&+K\int_0^1 z^{-1}\mmp\{\eta\leq \beta^{-1}xz\}(x-\beta)^{-(\alpha+\varepsilon)}{\rm d}z\\
&&+K\int_0^1 z^{-1}\int_{[1,x-\beta)}y^{-(\alpha+\varepsilon)}{\rm d}_y\mmp\{(1-y x^{-1})\eta\leq z\} {\rm d}z=:I_1(x)+I_2(x)+I_3(x).
\end{eqnarray*}
Firstly we calculate $I_2(x)$ explicitly as follows:
\begin{eqnarray*}
I_2(x)&=&K(x-\beta)^{-(\alpha+\varepsilon)}\int_0^1\mmp\{\eta\leq \beta^{-1}xz\}z^{-1}{\rm d}z\\
&=&K(x-\beta)^{-(\alpha+\varepsilon)}\int_0^{\beta x^{-1}}\mmp\{\eta\leq \beta^{-1}xz\}z^{-1}{\rm d}z+K(x-\beta)^{-(\alpha+\varepsilon)}(\log x-\log \beta)\\
&=&K(x-\beta)^{-(\alpha+\varepsilon)}\int_0^{1}\mmp\{\eta\leq z\}z^{-1}{\rm d}z+K(x-\beta)^{-(\alpha+\varepsilon)}(\log x-\log \beta)\\
&=&K(x-\beta)^{-(\alpha+\varepsilon)}(\me|\log\eta|+\log x-\log \beta))=O(x^{-(\alpha+\varepsilon)\log x}).
\end{eqnarray*}
Pick $\varepsilon'\in(0,\varepsilon]$ such that $\alpha+\varepsilon'<1$. The third summand $I_3(x)$ is estimated using the Fubini's theorem:
\begin{eqnarray*}
I_3(x)&\leq &K\int_0^1 z^{-1}\int_{[1,x-\beta)}y^{-(\alpha+\varepsilon')}{\rm d}_y\mmp\{(1-y x^{-1})\eta\leq z\} {\rm d}z\\
&=&K\int_0^1 z^{-1}\int_{[1,x-\beta)}y^{-(\alpha+\varepsilon')}\mmp\{(1-\eta^{-1}z)x\in {\rm d}y\}{\rm d}z\\
&=&K\int_0^1 z^{-1}\me \Big((1-\eta^{-1}z)x\Big)^{-(\alpha+\varepsilon')}1_{\{1\leq (1-\eta^{-1}z)x\leq x-\beta\}}{\rm d}z\\
&=&Kx^{-(\alpha+\varepsilon')} \me \int_0^1 z^{-1}(1-\eta^{-1}z)^{-(\alpha+\varepsilon')}1_{\{1\leq (1-\eta^{-1}z)x\leq x-\beta\}}{\rm d}z\\
&=&Kx^{-(\alpha+\varepsilon')} \me \int_{\beta\eta x^{-1}}^{\eta(1-x^{-1})} z^{-1}(1-\eta^{-1}z)^{-(\alpha+\varepsilon')}{\rm d}z\\
&\overset{z=\eta u}{=}&K x^{-(\alpha+\varepsilon')} \int_{\beta x^{-1}}^{1-x^{-1}} u^{-1}(1-u)^{-(\alpha+\varepsilon')}{\rm d}u=O(x^{-(\alpha+\varepsilon')}\log x).
\end{eqnarray*}

It remains to bound the first integral. To this end, note that for every $z\in[0,1)$ and $x\geq 1+\beta$,
$$\{1_{\{\theta_{\alpha}\geq x-\beta\}}+(1-\theta_{\alpha} x^{-1})\eta 1_{\{\theta_{\alpha}<x-\beta\}}\leq z\}=\{(1-\eta^{-1}z)x\leq \theta_{\alpha} < x-\beta\},$$
and therefore
\begin{eqnarray*}
&&\hspace{-1cm}\mmp\{1_{\{\theta_{\alpha}\geq x-\beta\}}+(1-\theta_{\alpha} x^{-1})\eta 1_{\{\theta_{\alpha}<x-\beta\}}\leq z\}\\
&=&\mmp\{(1-\eta^{-1}z)x\leq \theta_{\alpha} < x-\beta\}\\
&=&\mmp\{((1-\eta^{-1}z)x)\vee 1\leq \theta_{\alpha} < x-\beta\}\\
&=&\mmp\{\eta\leq \beta^{-1}xz,((1-\eta^{-1}z)x)\vee 1\leq \theta_{\alpha} < x-\beta\}\\
&=&\mmp\{\eta\leq \beta^{-1}xz,((1-\eta^{-1}z)x)\vee 1\leq \theta_{\alpha} < x\}-\mmp\{\eta\leq \beta^{-1}xz\}((x-\beta)^{-\alpha}-x^{-\alpha}).
\end{eqnarray*}
Putting this into $I_1(x)$ yields
\begin{eqnarray*}
I_1(x)&\leq&\int_0^1 |\mmp\{\eta\leq \beta^{-1}xz,((1-\eta^{-1}z)x)\vee 1\leq \theta_{\alpha} < x\}-\mmp\{\eta\leq z\}|z^{-1}{\rm d}z\\
&+&((x-\beta)^{-\alpha}-x^{-\alpha})\int_0^1 z^{-1}\mmp\{\eta\leq \beta^{-1}xz\}{\rm d}z.
\end{eqnarray*}
The second term is $O(x^{-\alpha-1}\log x)$ by the same argument as was used in the estimation of $I_2(x)$. Using simple algebra we obtain that the first term is equal to
\begin{eqnarray*}
&&\hspace{-1cm}\int_0^1\Big|\mmp\{z<\eta\leq(x-1)^{-1}xz\}\\
&+&x^{-\alpha}\Big(\int_{((x-1)^{-1}xz)\wedge 1}^{(\beta^{-1}xz)\wedge 1}((1-y^{-1}z)^{-\alpha})\mmp\{\eta\in{\rm d}y\}-\mmp\{\eta\leq \beta^{-1}xz\}\Big)\Big|z^{-1}{\rm d}z=:J(x).
\end{eqnarray*}
By the triangle inequality,
\begin{eqnarray}
&&\hspace{-1cm}J(x)\leq \int_0^1z^{-1}\mmp\{z<\eta\leq(x-1)^{-1}xz\}{\rm d}z\nonumber\\
&+&x^{-\alpha}\int_0^1\Big|\int_{((x-1)^{-1}xz)\wedge 1}^{(\beta^{-1}xz)\wedge 1}(1-y^{-1}z)^{-\alpha}\mmp\{\eta\in{\rm d}y\}-\mmp\{\eta\leq \beta^{-1}xz\}\Big|z^{-1}{\rm d}z.\label{4444}
\end{eqnarray}
The first summand, again by the Fubini's theorem, is calculated easily:
\begin{eqnarray*}
\int_0^1z^{-1}\mmp\{z<\eta\leq(x-1)^{-1}xz\}{\rm d}z&=&\me \int_0^1z^{-1}1_{\{z<\eta\leq(x-1)^{-1}xz\}}{\rm d}z\\
&=&\me \int_0^1z^{-1}1_{\{x^{-1}(x-1)\eta\leq z < \eta\}}{\rm d}z\\
&=&\me \int_{x^{-1}(x-1)\eta}^{\eta} z^{-1}{\rm d}z=|\log(1-x^{-1})|=O(x^{-1}).
\end{eqnarray*}
The inner integral in the second summand in rhs of \eqref{4444} is equal
\begin{eqnarray*}
\frac{\sin \pi\alpha}{\pi}\int_{((x-1)^{-1}xz)\wedge 1}^{(\beta^{-1}xz)\wedge 1}(y-z)^{-\alpha}(1-y)^{\alpha-1}{\rm d}y,
\end{eqnarray*}
and upon substitution $u:=(y-z)(1-z)^{-1}$ becomes
$$
\frac{\sin \pi\alpha}{\pi}\int_{\frac{z}{(1-z)(x-1)}\wedge 1}^{\frac{(\beta^{-1}x-1)z}{1-z}\wedge 1}u^{-\alpha}(1-u)^{\alpha-1}{\rm d}u=\mmp\Big\{\frac{z}{(1-z)(x-1)}\wedge 1\leq\eta\leq \frac{(\beta^{-1}x-1)z}{1-z}\wedge 1\Big\}.
$$
Since for $z\in[0,1)$ and $x>1+\beta$,
$$
0\leq \frac{z}{(1-z)(x-1)}\wedge 1\leq \frac{(\beta^{-1}x-1)z}{1-z}\wedge 1 \leq (\beta^{-1}xz)\wedge 1,
$$
the integral in the second summand in \eqref{4444} is
$$
\int_0^1 z^{-1} \mmp\Big\{\eta\leq \frac{z}{(1-z)(x-1)}\wedge 1\Big\}{\rm d}z+\int_0^1z^{-1}\mmp\Big\{\frac{(\beta^{-1}x-1)z}{1-z}\wedge 1\leq \eta\leq {(\beta^{-1}xz)\wedge 1}\Big\}{\rm d}z.
$$
We will check that the second summand above is $O(x^{-1})$ as follows:
\begin{eqnarray*}
&&\hspace{-3cm}\int_0^1z^{-1}\mmp\Big\{\frac{(\beta^{-1}x-1)z}{1-z}\wedge 1\leq \eta\leq {(\beta^{-1}xz)\wedge 1}\Big\}{\rm d}z\\
&&=\me \int_0^1 z^{-1} 1_{\{\eta\beta x^{-1}\leq z\leq \eta(\beta^{-1}x-1+\eta)^{-1}\}}{\rm d}z\\
&&=\me\Big(\log (\beta^{-1}x)-\log (\beta^{-1}x-1+\eta)\Big)=O(x^{-1}).
\end{eqnarray*}
The first term can be treated analogously, hence $J(x)=O(x^{-1})$. Combining all the estimates we get $s_{\theta}(x,\beta)=O(x^{-\alpha+\delta})$ for sufficiently small $\delta>0$. The proof is complete.
\end{proof}

\end{document}